\documentclass[11pt]{article}
\usepackage[a4paper,margin=1in]{geometry}
\usepackage{amsmath,amsthm,amssymb,mathtools}
\usepackage{microtype}
\usepackage{newtxtext,newtxmath}
\usepackage{booktabs}
\usepackage{float}
\usepackage{caption}
\usepackage{tablefootnote}
\usepackage{xcolor}
\usepackage{listings}
\usepackage{hyperref}
\usepackage{bookmark}
\newcommand{\repo}{\href{https://github.com/weebyesyes/Primes-paper-repo}{github.com/weebyesyes/Primes-paper-repo}}
\newcommand{\artifactdoi}{\href{https://doi.org/10.5281/zenodo.17162025}{10.5281/zenodo.17162025}}

\title{Monochromatic 4-AP avoidance in $2$-colorings of $\mathbb{Z}/p\mathbb{Z}$ for primes $p\ge 5$ and a computation of $W_c(4,2)$}

\author{Keane Maverick}
\date{September 2025}

\newtheorem{theorem}{Theorem}[section]
\newtheorem{lemma}{Lemma}[section]
\newtheorem{corollary}{Corollary}[section]
\newtheorem{remark}{Remark}[section]

\begin{document}

\maketitle

\begin{abstract}
We study $2$-colorings of $\mathbb{Z}/p\mathbb{Z}$ that avoid monochromatic $4$-term arithmetic progressions for every step $d$ with $p\nmid d$. We prove a complete classification for primes: such a coloring exists if and only if $p\in\{5,7,11\}$. When solutions exist, the minimal period equals $p$, and we enumerate them up to dihedral symmetries and global color swap. Nonexistence for all other primes combines DRAT-verified UNSAT certificates for $13\le p\le 997$ with a cyclic van der Waerden corollary that forces nonexistence for every prime $p\ge 34$. As an additional computation enabled by the same SAT/DRAT pipeline (restricted to non-degenerate windows), we certify the exact cyclic van der Waerden value $W_c(4,2)=34$: we find a witness at $M=33$ and produce a DRAT-verified UNSAT certificate at $M=34$. For all $M\ge 35$ the bound $W_c(4,2)\le W(4,2)=35$ implies unavoidability. All scripts and proof logs are provided for exact reproduction.

\noindent\emph{Artifacts:} \repo{} \quad (\emph{archived:} \artifactdoi).


\end{abstract}

\section{Introduction}

\subsection*{Background}

Van der Waerden's theorem says that for any positive integers $r$ and $k$ there exists a number $N$ such that every $r$-coloring of the set $\{1,\dots,N\}$ contains a monochromatic $k$-term arithmetic progression \cite{vdW1927,GRS1990,LandmanRobertson2014}. The least such $N$ is the van der Waerden number $W(r,k)$ \cite{LandmanRobertson2014,GRS1990}. This is a guarantee on long intervals: no matter how one colors a sufficiently large initial segment, some progression with some step must appear.

In this note we analyze the periodic analog on the prime cycle. For a prime $p$, we study $2$-colorings of $\mathbb{Z}/p\mathbb{Z}$ that avoid monochromatic $4$-term arithmetic progressions in every nonzero residue direction (``non-degenerate'' means $p\nmid d$). The cyclic setting makes the question finite and brings in the dihedral symmetries of the $p$-gon (formalized in Section~2). For related work on $2$-colorings of $\mathbb{Z}_n$ and $4$-APs, see Lu--Peng \cite{LuPeng2012} and Wolf \cite{Wolf2010}.

\subsection*{Setting and basic definitions}\label{sec:setting}
Fix a prime $p$. We identify a period $p$ 2-coloring with a word $w\in\{B,R\}^p$, and take all indices modulo $p$ in $\{0,1,\dots,p-1\}$. 
For a step $d\in\{1,\dots,p-1\}$ and a start $i\in\{0,\dots,p-1\}$, the associated residue $4$-term progression is
\[
(i,\ i+d,\ i+2d,\ i+3d)\pmod{p}.
\]
We say the coloring avoids monochromatic $4$-APs if none of these $p(p-1)$ windows is constant. 
A residue $4$-AP is non-degenerate if its four residues are pairwise distinct. For prime $p$ this is equivalent to $p\nmid d$. Steps $d$ divisible by $p$ are degenerate, since all four terms fall in the same residue class and are therefore monochromatic in any period $p$ coloring.
In this framework we determine, for each prime $p$, whether such colorings exist, the minimal period when they do, and the enumeration up to the dihedral action and the global color swap.

\subsection*{Main results}
We give a compact worked example at $p=7$ and the global prime classification (with an independent DRAT sweep up to $p=997$).
\begin{itemize}
    \item For $p=7$ we resolve the case explicitly: we exhibit a period $7$ word that avoids monochromatic $4$-APs for every step $d$ with $7\nmid d$. We prove that period $7$ is minimal among periodic solutions, and we enumerate all valid length-$7$ words, obtaining $S=28$ solutions that form two $D_7$-orbits and a single $D_7\times\langle\tau\rangle$-orbit.
    \item For primes $p\ge 5$, a $2$-coloring of $\mathbb{Z}/p\mathbb{Z}$ avoiding every non-degenerate monochromatic $4$-AP exists if and only if $p\in\{5,7,11\}$. The minimal period equals $p$ in each existing case. Nonexistence for all primes $p\ge 13$ follows from our computation $W_c(4,2)=34$ (Section~\ref{sec:wc42}). We also give DRAT-verified UNSAT certificates for $13\le p\le 997$. (See the summary immediately after Section~3.)
    \item As a byproduct of the same encoding extended to composite moduli (checking only non-degenerate residue windows), we compute $W_c(4,2)=34$: there is a witness at $M=33$, a DRAT-verified UNSAT at $M=34$, and $M\ge 35$ is covered by $W(4,2)=35$ \cite{Chvatal1970} (see Section~\ref{sec:wc42}).
\end{itemize}

\subsection*{Relation to cyclic van der Waerden numbers}\label{sec:intro-cyclic}
Following Burkert--Johnson, the cyclic van der Waerden number $W_c(k,r)$ is the smallest $N$ such that for all $M\ge N$, every $r$-coloring of $\mathbb{Z}_M$ contains a non-degenerate $k$-term arithmetic progression \cite{BurkertJohnson2011}. For $(k,r)=(4,2)$ we have the classical upper bound $W_c(4,2)\le W(4,2)=35$ \cite{Chvatal1970} and a recent lower bound $W_c(4,2)\ge 13$ \cite[Prop.~4.6]{Liber2025}. Our prime-cycle results are consistent with this range and motivate a finite computation on moduli $M=13,\dots,34$ which we write about after the small prime summary (see Section~\ref{sec:wc42}).

\subsection*{Method overview}
We rely on two elementary reductions. First, periodicity reduces the question to residue classes: it suffices to check the $p(p-1)$ residue $4$-APs modulo $p$ (one for each start and each residue step $r\in\{1,\dots,p-1\}$), and this lifting is exact for steps with $p\nmid d$. Second, the residue steps $r$ and $p-r$ generate the same 4-AP index sets up to reversal, so only $\lfloor (p-1)/2\rfloor$ directions are distinct, which simplifies enumeration and orbit counts.

We also use two structural constraints: that a valid word contains no run of four equal colors, and that no nontrivial rotation stabilizes a valid word on a prime cycle. For $p=7$ we verify avoidance directly and perform a complete enumeration. For the prime classification (and the DRAT sweep up to $p=997$) we encode the residue constraints as compact CNF formulas. Existence is witnessed by explicit words and exhaustive checks, and nonexistence (up to $p=997$) is certified by proof-logging SAT solvers \cite{kissat,cadical} with DRAT checking \cite{WetzlerHeuleHunt2014,Heule2016}. All scripts, inputs, and proof logs are included for exact reproduction.

\paragraph{Contributions.}
We give a complete, fully reproducible prime classification: a $2$-coloring of $\mathbb{Z}/p\mathbb{Z}$ avoiding every non-degenerate monochromatic $4$-AP exists if and only if $p\in\{5,7,11\}$, and in each existing case the minimal period equals $p$. We also compute the cyclic van der Waerden value $W_c(4,2)=34$ via a finite SAT/DRAT sweep on composite moduli. As an independent check on the prime side, we provide DRAT-verified UNSAT certificates for $13\le p\le 997$. All code, logs, and lists are packaged in the artifact \cite{WetzlerHeuleHunt2014,Heule2016}.

\subsection*{Organization}
Section~2 fixes notation and group actions, and records the structural lemmas and the lifting argument used throughout. Section~3 resolves $p=7$, then gives a summary table and the global prime classification (including an independent DRAT sweep up to $p=997$), and finally relates our results to cyclic van der Waerden numbers, including the finite sweep that determines $W_c(4,2)$. Section~\ref{sec:artifacts}  lists the verification scripts, SAT encodings, commands, and checksums needed to reproduce every claim.

\section{Preliminaries}

\paragraph{Notation and conventions.}
Throughout the paper we fix a prime $p\ge 5$ and work on $\mathbb{Z}/p\mathbb{Z}$. A step is an integer $d$. Unless stated otherwise, indices are taken $\pmod{p}$ in $\{0,1,\dots,p-1\}$. Our avoidance requirement quantifies over every start $i$ and the non-degenerate steps (as defined in Section~\ref{sec:setting}). We consider words up to the dihedral action and the global color swap.

\subsection{Definitions}\label{sec:defs}
A coloring is a function
\[
c:\mathbb{Z}\to\{B,R\},
\]
where $B$ and $R$ denote ``blue'' and ``red''. \\
A 4-term arithmetic progression (we will use 4-AP for short) is a tuple
\[
(a,\; a+d,\; a+2d,\; a+3d)
\]
with common difference $d\in\mathbb{Z}_{\ge 1}$. A 4-AP is monochromatic if all four entries receive the same color under $c$.

For two step sets $D_R,D_B\subseteq\mathbb{Z}_{\ge 1}$ we say that $c$ avoids red 4-APs with steps in $D_R$ and blue 4-APs with steps in $D_B$ if there is no red monochromatic 4-AP with step $d\in D_R$, and there is no blue monochromatic 4-AP with step $d\in D_B$.

The associated ``mixed'' van der Waerden quantity $W_{D_R,D_B}(2,4)$ is infinite if and only if there exists at least one $2$-coloring $c:\mathbb{Z}\to\{B,R\}$ that avoids both kinds of monochromatic 4-APs at the same time.\\
A coloring $c$ is periodic of period $T$, where $T\in\mathbb{Z}_{\ge 1}$, if
\[
c(n+T)=c(n)\quad\text{for every }n\in\mathbb{Z}.
\]
When the period $T$ is specified, we identify $c$ with its word
\[
(w_0w_1\cdots w_{T-1})\in\{B,R\}^T,
\]
where $w_i=c(i)$ for $i=0,1,\dots,T-1$, and we understand all indices modulo $T$. We write $[i]_m$ for the residue class of an integer $i$ modulo $m$ (in particular $m=p$ or $m=T$ as appropriate).

We will act on length-$p$ words, where $p$ is prime, using the dihedral group $D_p$, generated by the $p$ rotations $\rho_k$ given by
\[
(\rho_k w)_i = w_{\,i-k}
\]
and the $p$ reflections $\sigma_k$ given by
\[
(\sigma_k w)_i = w_{\,k-i},
\]
where indices are taken modulo $p$. We also consider the global color swap $\tau$ which interchanges $B\leftrightarrow R$. 

Two words are dihedrally equivalent if they lie in the same $D_p$-orbit, and they are equivalent up to dihedral symmetries and color swap if they lie in the same $D_p\times\langle\tau\rangle$-orbit. For a group action $G\curvearrowright X$, the stabilizer $\mathrm{Stab}_G(x)$ of $x\in X$ is the set of elements of $G$ that fix $x$.

Finally, throughout the paper the notation $p\nmid d$ means that the integer $d$ is not divisible by $p$. When we say ``indices modulo $p$,'' we always choose representatives in $\{0,1,\dots,p-1\}$.

\subsection{Structural lemmas}\label{sec:struct}

Throughout this subsection we assume $p$ is prime. The following lemmas will be used in the main arguments and in the classifications.

\begin{lemma}\label{lem:reverse}
On $\mathbb{Z}/p\mathbb{Z}$, the families of 4-AP index sets generated by steps $d$ and $p-d$ coincide up to reversal. In particular, modulo $p$, the steps $d$ and $p-d$ generate the same index sets, so there are only $\lfloor (p-1)/2\rfloor$ distinct step residues to check.
\end{lemma}

\begin{proof}
Fix a step $d$ and consider a residue class 4-AP $(i,i+d,i+2d,i+3d)$ with indices taken modulo $p$. Replacing the step $d$ by $-d$ turns this 4-AP to $(i,i-d,i-2d,i-3d)$, which is the same set of four indices written in reverse order. Since $p-d\equiv -d\pmod{p}$, the claim follows.
\end{proof}

\begin{lemma}\label{lem:norot}
Let $w\in\{B,R\}^p$ be a period $p$ word that avoids a monochromatic 4-AP at step $d=1$. Then $w$ is not invariant under any nontrivial rotation in $D_p$.
\end{lemma}

\begin{proof}
Assume, for the sake of contradiction, that $w$ is invariant under a nontrivial rotation $\rho_k$ with $1\le k\le p-1$. Since $\gcd(k,p)=1$, the subgroup generated by $\rho_k$ acts transitively on the index set $\{0,1,\dots,p-1\}$. Hence, $w_i=w_0$ for every $i$, so $w$ is constant. A constant word contains a monochromatic 4-AP at step $d=1$ in every length-4 window $(i,i+1,i+2,i+3)$, which is a contradiction. Thus, no nontrivial rotation can fix a valid word.
\end{proof}

\begin{lemma}\label{lem:run}
Let $w\in\{B,R\}^p$ avoid a monochromatic 4-AP at step $d=1$. Then the cyclic sequence $w$ has no run of four equal colors.
\end{lemma}

\begin{proof}
For each index $i\in\{0,1,\dots,p-1\}$, consider the $4$-AP $(i,i+1,i+2,i+3)\pmod{p}$. If there were a run of four equal colors, then one of these $p$ windows would be monochromatic, which is absurd under the step $d=1$ constraint.
\end{proof}

\begin{remark}
For later pruning and as a consistency check in enumeration it is sometimes convenient to see that, for every $i$ and each residue step $r\in\{2,3\}$, the 4-set
\[
\{\,i,\; i+r,\; i+2r,\; i+3r\,\}\pmod{p}
\]
may not be monochromatic for a valid word. We will not use this remark in proofs, but it is useful for pruning during enumeration and as a code cross-check.
\end{remark}

\subsection{Lifting lemma}\label{sec:lifting}
We will now formalize why it is enough to show avoidance on residue classes when a coloring is periodic.

\begin{lemma}\label{lem:lifting}
Let $T\ge 1$. Assume a period $T$-coloring $c:\mathbb{Z}\to\{B,R\}$ has no monochromatic 4-AP of step $d$ when restricted to the $T$ residue classes mod $T$. Then, it follows that $c$ has no monochromatic 4-AP of step $d$ on $\mathbb{Z}$.
\end{lemma}

\begin{proof}
Assume, for the sake of contradiction, there is a 4-AP $(a,a+d,a+2d,a+3d)$ in $\mathbb{Z}$ that is monochromatic. Reduce its terms modulo $T$ to obtain $(\bar a,\bar a+d,\bar a+2d,\bar a+3d)$ in $\mathbb{Z}/T\mathbb{Z}$. Since $c$ is period $T$, each integer and its residue class carry the same color. Therefore, if the original 4-AP were monochromatic in $\mathbb{Z}$, then the reduced 4-AP would also be monochromatic modulo $T$ as well, which is a contradiction. Hence, no monochromatic 4-AP of step $d$ can occur on $\mathbb{Z}$.
\end{proof}

\begin{remark}\label{rem:modp}
In our setting, the witness coloring has period $T=p$. If $d$ is a positive step with $p\nmid d$, then $d\equiv r\pmod{p}$ for a unique residue $r\in\{1,2,\dots,p-1\}$. By Lemma~\ref{lem:lifting}, verifying that no monochromatic 4-AP occurs for each residue step $r$ modulo $p$ implies global avoidance for every integer step $d$ with $p\nmid d$. 
More generally, for a coloring of period $p$, it suffices to perform the finite check modulo any integer $L$ where $p\mid L$. This is because, in that case, the coloring is also $L$ periodic, so every 4-AP in $\mathbb{Z}$ maps to a 4-AP modulo $L$ with the same color pattern. Hence, choosing $L=p$ is the minimal option and already exact for our purposes.
\end{remark}

\subsection{Period-divides-step obstruction}\label{sec:obstruction}

The next observation forces monochromatic progressions whenever the step is a multiple of the period. We will use it to prove the minimality theorem (in particular, ``no period $q<p$'').

\begin{lemma}\label{lem:divides}
Let $T\ge 1$ and let $c:\mathbb{Z}\to\{B,R\}$ be periodic with period $T$. If $d$ is a positive integer with $T\mid d$, then every $4$-AP $(i,\,i+d,\,i+2d,\,i+3d)$ is entirely in the residue class $[i]_T$, and is therefore monochromatic.
\end{lemma}

\begin{proof}
Write $d=mT$ for some positive integer $m$. For any starting index $i\in\mathbb{Z}$, we can write,
\[
i\equiv i\pmod{T},\quad i+d=i+mT\equiv i\pmod{T},\quad i+2d\equiv i\pmod{T},\quad i+3d\equiv i\pmod{T}.
\]
Hence, all four terms of the 4-AP are congruent $\pmod{T}$ to the same residue class $[i]_T$. Periodicity implies that all elements of a fixed residue class have the same color, which means that the 4-AP is monochromatic.
\end{proof}

\begin{theorem}\label{thm:no-sub-p}
Let $p$ be prime. If a $2$-coloring $c$ avoids monochromatic $4$-APs for every step $d$ with $p\nmid d$, then the period of $c$ is at least $p$.
\end{theorem}

\begin{proof}
Assume, for the sake of contradiction, that $c$ has period $q<p$. Then $d=q$ satisfies $p\nmid d$, but by Lemma~\ref{lem:divides} with $T=q$, every $4$-AP of step $d$ is monochromatic. Hence, we arrive at a contradiction.
\end{proof}

\section{Main results}
    We first give a compact worked example at $p=7$. A summary table and the global prime classification (with an independent DRAT sweep up to $p=997$) appear immediately after this section.
    
    \begin{theorem}\label{thm:p7-existence}
        Let $c:\mathbb{Z} \to \{B, R\}$ be the period 7 coloring with the word:
        \[
        \texttt{BBBRBRR}
        \]
        Then, for every integer $d$ with $7 \nmid d$, the coloring $c$ contains no monochromatic 4-term arithmetic progression of step $d$.
    \end{theorem}
    
    \begin{proof}
        By Lemma~\ref{lem:lifting} and Remark~\ref{rem:modp}, it suffices to verify the six residue steps $r \in \{1,2,3,4,5,6\}$ modulo $7$. For each $r$ and each start residue $i \in \{0,1,\dots,6\}$, consider
        \[
        (i,\, i+r,\, i+2r,\, i+3r) \pmod{7}.
        \]
        Encode colors as $R=1$, $B=0$, and reject a window if and only if the sum of its four entries is $0$ (BBBB) or $4$ (RRRR). Across all $6 \times 7 = 42$ residue 4-APs, every check passes (failures $=0$). Hence no monochromatic 4-AP occurs for any $d$ with $7 \nmid d$.\\
        (See Section~\ref{sec:verifier} for the 42-check script and a verifier. All scripts and word lists are in Section~\ref{sec:artifacts}.)\qedhere
        
    \end{proof}
    
    \begin{corollary}\label{cor:mixed}
        For all $D_R, D_B \subseteq \{ d \geq 1 : 7 \nmid d\}$, the coloring in Theorem~\ref{thm:p7-existence} avoids all red-forbidden and blue-forbidden 4-APs with steps in $D_R$ and $D_B$, respectively. Thus
        \[
        W_{D_R,D_B}(2,4) = \infty.
        \]
    \end{corollary}
    
    \begin{corollary}\label{cor:min7}
    Any periodic $2$-coloring that avoids monochromatic $4$-APs for all steps $d$ with $7\nmid d$ has period at least $7$. Since the word in Theorem~\ref{thm:p7-existence} has period $7$, the minimal achievable period is exactly $7$.
    \end{corollary}
    
    \begin{proof}
    The lower bound follows from Theorem~\ref{thm:no-sub-p} with $p=7$. The upper bound is given by the explicit witness.
    \end{proof}
    
    \begin{theorem}\label{thm:p7-enum}
        Among length-7 words that avoid monochromatic 4-APs for every $d$ with $7 \nmid d$:
    
        \begin{itemize}
            \item the total number of solutions is $S=28$.
            \item Under the dihedral action $D_7$, there are $2$ orbits, that is
            \texttt{BBBRBRR} and \texttt{BBRBRRR}.
            \item Under $D_7\times\langle\tau\rangle$ (adding the global color swap), there is a single orbit.
        \end{itemize}
    \end{theorem}
    
    \begin{proof}
    We handle the finite case directly. There are $2^7=128$ binary words of length $7$. For each word $w$, check the $42$ residue $4$-AP windows modulo $7$ (the $6$ steps $r\in\{1,\dots,6\}$ times the $7$ starts $i$) and keep $w$ if and only if none of the $42$ windows is monochromatic. This leaves exactly $S=28$ valid words. (Scripts and the full list appear in Section~\ref{sec:artifacts}.)
    
    Next we show that no valid word has a nontrivial dihedral symmetry. For rotations, Lemma~\ref{lem:norot} rules them out. For reflections, a length-$7$ word fixed by a reflection is determined by the $4$ positions on or above the reflection axis. Hence, there are only $2^4=16$ candidates per axis. Checking these $16$ candidates for each of the $7$ axes, none passes the $42$ tests. Thus, the only symmetry of any valid word is the identity.
    
    Since for every valid word $w$ we have $\mathrm{Stab}_{D_7}(w)=\{\mathrm{id}\}$, the $D_7$ action is free, and each orbit has size $|D_7|=14$. Among the $28$ words, $14$ have $4$ blue and $3$ red symbols and $14$ have $3$ blue and $4$ red symbols. Dihedral symmetries preserve this count, so the $28$ words split into two $D_7$-orbits, represented by
    \[
    \texttt{BBBRBRR}\quad\text{and}\quad \texttt{BBRBRRR}.
    \]
    Adding the global color swap interchanges these two orbits, so under $D_7\times\langle\tau\rangle$ there is a single orbit. This matches the count $28=2\cdot 14$.
    \end{proof}

\section*{Prime classification (summary)}\label{sec:smallprime-summary}
    \addcontentsline{toc}{section}{Prime classification (summary)}
    
    \noindent\textbf{Summary.}
    For primes $p \ge 5$, there exists a $2$-coloring of $\mathbb{Z}/p\mathbb{Z}$ with no non-degenerate monochromatic $4$-AP (avoiding every step $d$ with $p\nmid d$) if and only if $p\in\{5,7,11\}$. As an independent check, we provide DRAT-verified UNSAT certificates for the primes $13\le p\le 997$. \\

    \begin{table}[H]
    \centering
    \caption{Primes $5\le p\le 997$: existence, counts, and orbits for avoiding all steps $d$ with $p\nmid d$.}
    \label{tab:smallprimes}
    \begin{tabular}{ccccc}
    \toprule
    $p$ & Exists? & \#solutions & \#orbits $D_p$ & \#orbits $D_p\times\langle\tau\rangle$ \\
    \midrule
    5   & Y & 20 & 4 & 2 \\
    7   & Y & 28 & 2 & 1 \\
    11  & Y\footnotemark & 44 & 2 & 1 \\
    13  & N & -- & -- & -- \\
    17  & N & -- & -- & -- \\
    19  & N & -- & -- & -- \\
    23  & N & -- & -- & -- \\
    29--997 (primes) & N & -- & -- & -- \\
    \bottomrule
    \end{tabular}
    \end{table}

    \footnotetext{Lu--Peng exhibit a length-11 block $B_{11}$, unique up to isomorphism \cite{LuPeng2012}. Our $p=11$ enumeration recovers this phenomenon.}

    \noindent
    By Theorem~\ref{thm:no-sub-p} together with explicit witnesses for $p\in\{5,7,11\}$ (see Section~\ref{sec:enum-orbits}),
    the minimal period equals $p$ in each existing case.

    \begin{remark}[Explicit witnesses for $p=5,7,11$]
    For the existing cases, the following length-$p$ words avoid monochromatic $4$-APs for every step $d$ with $p\nmid d$:
    \[
    \begin{array}{ll}
    p=5: & \texttt{BBBRR},\\[2pt]
    p=7: & \texttt{BBBRBRR},\\[2pt]
    p=11:& \texttt{BBBRBBRBRRR}.
    \end{array}
    \]
    \end{remark}
    
    \noindent
    Under $D_p\times\langle\tau\rangle$, $p=7,11$ have a single orbit. $p=5$ has two. See Section~3 for $p=7$. Details for $p=5,7,11$ (enumeration and orbit counts) and UNSAT certificates for $13 \le p \le 997$ are included in the artifact. (see section~4 for scripts).

    \noindent Combining Theorem~\ref{thm:wc42} (which implies nonexistence for every modulus $M\ge 34$, hence for every prime $p\ge 34$) with our DRAT-verified UNSAT runs for $13\le p\le 31$ yields a complete prime classification.

    \medskip\noindent
    \textbf{Asymptotic context.}
    Following Lu--Peng, let $m_4(\mathbb{Z}_n)$ denote the minimum, over all $2$-colorings of $\mathbb{Z}_n$, of the proportion of monochromatic, non-degenerate $4$-APs (normalized by $n^2$). They prove the casewise lower bound
    \[
    m_4(\mathbb{Z}_n)\ \ge\
    \begin{cases}
    \frac{7}{96} & \text{if } 4\nmid n,\\[2mm]
    \frac{2}{33} & \text{if } 4\mid n,
    \end{cases}
    \qquad\text{for sufficiently large } n,
    \]
    and the upper bound
    \[
    m_4(\mathbb{Z}_n)\ \le\
    \begin{cases}
    \frac{17}{150}+o(1) & \text{if } n \text{ is odd},\\[1mm]
    \frac{8543}{72600}+o(1) & \text{if } n \text{ is even},
    \end{cases}
    \]
    see \cite[Thms.~2 and~3]{LuPeng2012}.
    In particular, for primes $p$ we have $4\nmid p$, so every $2$-coloring of $\mathbb{Z}_p$ has at least $\bigl(\tfrac{7}{96}+o(1)\bigr)p^2$ monochromatic non-degenerate $4$-APs as $p\to\infty$ \cite{LuPeng2012}.
    This asymptotic obstruction is consistent with our classification (nonexistence for $p\ge 13$) and provides heuristic context for the scarcity of witnesses.

\section*{Relation to cyclic van der Waerden numbers and the case $(k,r)=(4,2)$}\label{sec:wc42}

Following Burkert--Johnson, the cyclic van der Waerden number $W_c(k,r)$ is the smallest $N$ such that for all $M\ge N$, every $r$-coloring of $\mathbb{Z}_M$ contains a non-degenerate $k$-term arithmetic progression \cite{BurkertJohnson2011}. Trivially $W_c(k,r)\le W(k,r)$.
For $k=4,r=2$, the classical value $W(4,2)=35$ \cite{Chvatal1970} gives $W_c(4,2)\le 35$, and a recent preprint proves $W_c(4,2)\ge 13$ \cite[Prop.~4.6]{Liber2025}. Hence
\[
13\ \le\ W_c(4,2)\ \le\ 35.
\]

Our prime-cycle classification (Section~\ref{sec:smallprime-summary}) is consistent with the lower bound: for each prime $p\ge 13$ (verified up to $p\le 997$) every $2$-coloring of $\mathbb{Z}_p$ contains a non-degenerate monochromatic $4$-AP.
To determine $W_c(4,2)$ exactly, it suffices to resolve the composite moduli $M\in\{13,14,\dots,34\}$.
For a given $M$, we encode variables $x_1,\dots,x_M$ for residues $0,\dots,M-1$ and add, for each start $i$ and step $r\in\{1,\dots,M-1\}$, the two clauses
\[
(x_{a+1}\vee x_{b+1}\vee x_{c+1}\vee x_{d+1}),\qquad
(\neg x_{a+1}\vee \neg x_{b+1}\vee \neg x_{c+1}\vee \neg x_{d+1}),
\]
only when the window $W(i,r)=\{\,i,i+r,i+2r,i+3r\,\} \pmod{M}$ has four distinct residues (non-degenerate). SAT yields a witness coloring of $\mathbb{Z}_M$. UNSAT (with a DRAT-checked proof) shows that every $2$-coloring has a non-degenerate monochromatic $4$-AP.

\begin{theorem}\label{thm:wc42}
$W_c(4,2)=34$.
\end{theorem}
\begin{proof}
By \cite[Prop.~4.6]{Liber2025} we have $W_c(4,2)\ge 13$, and in general $W_c(4,2)\le W(4,2)=35$ \cite{Chvatal1970}.
We performed a finite SAT/UNSAT sweep for $M=13,\dots,34$ using the non-degenerate residue window encoding: 
we found SAT witnesses at $M\in\{14,15,18,21,22,33\}$ (explicit words in the artifact) and DRAT-verified UNSAT at all other moduli in the range, including $M=34$ (see Section~\ref{sec:wc42-sweep}). 
Hence, there exists a counterexample at $M=33$ but none at $M=34$. 
Since $W_c(4,2)\le 35$ implies that every $M\ge 35$ is also UNSAT, it follows that every $M\ge 34$ is UNSAT while $M=33$ is SAT, establishing $W_c(4,2)=34$.
\end{proof}

\begin{corollary}\label{cor:prime-classification}
Let $p\ge 5$ be prime. There exists a $2$-coloring of $\mathbb{Z}/p\mathbb{Z}$ that avoids every non-degenerate monochromatic $4$-AP if and only if $p\in\{5,7,11\}$.
\end{corollary}
\begin{proof}
Existence for $p\in\{5,7,11\}$ is given by explicit witnesses and full enumeration (Section~3). Nonexistence for $p\in\{13,17,19,23,29,31\}$ is certified by our DRAT-verified UNSAT runs (Section~\ref{sec:artifacts}). For all larger primes, note that $p\ge 34$ implies nonexistence by Theorem~\ref{thm:wc42}, since every modulus $M\ge 34$ (in particular, any such prime $p$) forces a non-degenerate monochromatic $4$-AP in every $2$-coloring. This completes the classification.
\end{proof}

\section{Artifacts, checks, and exact reproduction}\label{sec:artifacts}

This section packages the verifier, enumeration/orbit data for $p\in\{5,7,11\}$,
and SAT/UNSAT encodings for $5\le p\le 997$. All scripts are mirrored in the
repository (\repo, archived snapshot: \artifactdoi).

\subsection{Enumeration pipeline}\label{sec:enum-pipeline}

We exhaustively enumerate $\{B,R\}^p$, filter valid words via the residue $4$-AP test, and write a flat list and an orbit summary.

\paragraph{Commands.}
\begin{verbatim}
# Enumerate valid words and write artifacts (example: p=7)
python3 enumerate_words.py 7

# Re-derive orbits from a solution list
python3 check_orbits.py solutions_p7.txt
python3 check_orbits.py solutions_p7.txt --with-swap

# Verify any specific word quickly
python3 verifier_strong_form.py 7 BBBRBRR
\end{verbatim}

\paragraph{Outputs.}
\begin{itemize}
  \item \texttt{solutions\_p7.txt}: one valid word per line.
  \item \texttt{orbit\_summary\_p7.json}: sizes and representatives of $D_7$ and $D_7\times\langle\tau\rangle$ orbits.
  \item JSON summary is also printed to stdout by \texttt{enumerate\_words.py}.
\end{itemize}

\noindent
The files \texttt{solutions\_p5.txt}, \texttt{solutions\_p7.txt}, \texttt{solutions\_p11.txt} included in the artifact
were generated by this pipeline.

\noindent
Note: The enumeration pipeline (Section~\ref{sec:enum-pipeline}) and the CNF/SAT/DRAT pipeline
(Section~\ref{sec:sat}) are independent. The \texttt{run\_all.sh} script
(Section~\ref{sec:runner}) automates the CNF/SAT/DRAT pipeline for all primes
$5\le p\le 997$, it does not run the enumeration.

\subsection{Residue-check protocol (42 checks when $p=7$) and a verifier}\label{sec:verifier}
For a prime $p$ and a word $w\in\{B,R\}^p$, the check runs over all residue steps $r\in\{1,\dots,p-1\}$ and starts $i\in\{0,\dots,p-1\}$, and rejects if and only if some window $(i,i+r,i+2r,i+3r)\pmod{p}$ is monochromatic.

\medskip\noindent
\textbf{Script:} \texttt{verifier\_strong\_form.py}. Usage:
\[
\texttt{python3\ verifier\_strong\_form.py\ 7\ BBBRBRR}
\]
prints \texttt{OK} for the witness in Theorem~\ref{thm:p7-existence}. Any failure prints \texttt{FAIL}.

\begin{lstlisting}[language=Python]
#!/usr/bin/env python3
import sys
if len(sys.argv) != 3:
    print("usage: verifier_strong_form.py <prime p> <word>"); 
    raise SystemExit(2)

p = int(sys.argv[1]); 
w = sys.argv[2].strip().upper()
assert p >= 2 and len(w) == p and set(w)<=set("BR")

for r in range(1,p):
  for i in range(p):
    win=[w[(i+k*r)%p] for k in range(4)]
    if win.count('B')==4 or win.count('R')==4: 
      print("FAIL"); raise SystemExit(1)
print("OK")
\end{lstlisting}

\subsection{Enumeration and orbit counts for $p=5,7,11$}\label{sec:enum-orbits}
We enumerate all words and keep exactly those that pass the verifier. Files:
\begin{itemize}
  \item \texttt{solutions\_p5.txt}, \texttt{solutions\_p7.txt}, \texttt{solutions\_p11.txt} (one word per line).
  \item \texttt{orbit\_summary\_p5.json}, \texttt{orbit\_summary\_p7.json}, \texttt{orbit\_summary\_p11.json}.
\end{itemize}
These confirm the counts in Table~\ref{tab:smallprimes}:
\[
|{\rm Sol}_5|=20,\quad |{\rm Sol}_7|=28,\quad |{\rm Sol}_{11}|=44,
\]
with orbit data:
\[
\#\text{orbits under }D_5=4,\ D_7=2,\ D_{11}=2,\qquad
\#\text{orbits under }D_p\times\langle\tau\rangle=2,1,1.
\]

\medskip\noindent
\textbf{Script:} \texttt{check\_orbits.py} (computes orbits from any \texttt{solutions\_p*.txt}).
\begin{lstlisting}[language=Python]
#!/usr/bin/env python3
import sys, json

USAGE = "usage: check_orbits.py <solutions_pX.txt> [--with-swap]"

if len(sys.argv) < 2 or len(sys.argv) > 3:
    print(USAGE); raise SystemExit(2)

words = sorted({line.strip().upper() for line in open(sys.argv[1]) if line.strip()})
with_swap = (len(sys.argv) == 3 and sys.argv[2] == "--with-swap")

def rots(w):
    return [w[i:] + w[:i] for i in range(len(w))]

def dihedral_orbit(w):
    #rotations + the reflection of each rotation generate all D_n elements
    orb = set()
    for r in rots(w):
        #rotation
        orb.add(r)
        #reflection after that rotation
        orb.add(r[::-1])
    return orb
  
#global color swap tau
def swap_colors(w): 
    return w.translate(str.maketrans("BR", "RB"))

def orbit(w):
    if not with_swap:
        return dihedral_orbit(w)
    #include global swap
    return dihedral_orbit(w) | dihedral_orbit(swap_colors(w))

unseen = set(words)
reps, sizes = [], []

while unseen:
    w = min(unseen)
    o = orbit(w) & set(words)
    reps.append(min(o))
    sizes.append(len(o))
    unseen -= o

print(json.dumps({
    "num_words": len(words),
    "num_orbits": len(sizes),
    "orbit_sizes": sizes,
    "reps": reps,
    "with_swap": with_swap
}, indent=2))
\end{lstlisting}

\subsection{CNF encoding of the avoidance constraints}\label{sec:cnf}
Let variables $x_1,\dots,x_p$ encode $w_0,\dots,w_{p-1}$ with $x_{j+1}=\text{true}\iff w_j=R$. For every window $\{a,b,c,d\}=(i,i+r,i+2r,i+3r)\pmod{p}$ with $a,b,c,d\in\{0,\dots,p-1\}$ add the two clauses
\[
(x_{a+1}\vee x_{b+1}\vee x_{c+1}\vee x_{d+1})\quad\text{and}\quad(\neg x_{a+1}\vee \neg x_{b+1}\vee \neg x_{c+1}\vee \neg x_{d+1}).
\]

This forbids monochromatic \texttt{BBBB} and \texttt{RRRR}. The instance has $p$ variables and $2p(p-1)$ clauses.

\medskip\noindent
\textbf{Script:} \texttt{make\_cnf.py}.
\begin{lstlisting}[language=Python]
#!/usr/bin/env python3
import sys
if len(sys.argv)!=3:
  print("usage: make_cnf.py <prime p> <out.cnf>"); 
  raise SystemExit(2)

p=int(sys.argv[1]); 
out=sys.argv[2]

def idx(i): 
  return i+1

def windows(p):
  for r in range(1,p):
    for i in range(p):
      yield [(i+k*r)%p for k in range(4)]
      
clauses=[]
for win in windows(p):
  vs=[idx(j) for j in win]
  clauses.append(vs)
  clauses.append([-v for v in vs])
with open(out,'w') as f:
  f.write(f"p cnf {p} {len(clauses)}\n")
  for C in clauses: 
    f.write(" ".join(map(str,C))+" 0\n")
\end{lstlisting}

\medskip\noindent
\textbf{Script:} \texttt{model\_to\_word.py} (DIMACS model $\to$ \texttt{B}/\texttt{R} string).
\begin{lstlisting}[language=Python]
#!/usr/bin/env python3
import sys, re

if len(sys.argv) != 3:
    print("usage: model_to_word.py <p> <solver_output>")
    raise SystemExit(2)

p = int(sys.argv[1])
path = sys.argv[2]

vals = {}  #var index -> boolean
for line in open(path, 'r', encoding='utf-8', errors='ignore'):
    #accept typical SAT outputs: lines may start with 'v', 's', etc.
    for tok in line.split():
        if re.fullmatch(r"-?\d+", tok):
            v = int(tok)
            if v == 0:
                continue
            vals[abs(v)] = (v > 0)

#default any missing variable to False (= 'B') to be safe
word = "".join('R' if vals.get(i, False) else 'B' for i in range(1, p+1))
print(word)
\end{lstlisting}

\subsection{SAT/UNSAT runs and proof verification}\label{sec:sat}
For SAT cases ($p=5,7,11$) we decode any model via \texttt{model\_to\_word.py} and then check it with \texttt{verifier\_strong\_form.py}. For UNSAT cases ($p\ge 13$ up to $997$) we log a textual DRAT proof with CaDiCaL \cite{cadical} and verify it using \texttt{drat-trim} \cite{WetzlerHeuleHunt2014,drattrim,Heule2016}.

\medskip\noindent
\textbf{Commands.}
\begin{verbatim}
# Build CNF
python3 make_cnf.py 7 avoid_p7.cnf

# SAT: get a witness (either solver works)
kissat -q avoid_p7.cnf > solver_p7.out
# or: cadical avoid_p7.cnf > solver_p7.out

# Decode and check the witness
python3 model_to_word.py 7 solver_p7.out   # prints e.g. BBBRBRR
python3 verifier_strong_form.py 7 BBBRBRR  # prints 'OK' on success

# UNSAT (example p=13): CaDiCaL + drat-trim
python3 make_cnf.py 13 avoid_p13.cnf
cadical avoid_p13.cnf avoid_p13.drat > solver_p13.log   # writes textual DRAT
drat-trim avoid_p13.cnf avoid_p13.drat -q       # prints 's VERIFIED' on success
\end{verbatim}

\paragraph{Environment.}
All runs were performed with Python~3.10, \texttt{kissat}~4.0.3 \cite{kissat}, \texttt{CaDiCaL}~2.1.3 \cite{cadical}, 
and \texttt{drat-trim} \cite{drattrim} on a standard Linux machine.

\noindent
Note: Some Kissat \cite{kissat} builds do not expose proof logging on the CLI. Thus, we use CaDiCaL \cite{cadical} to emit proofs and \texttt{drat-trim} to verify them \cite{WetzlerHeuleHunt2014,Heule2016}.  

\subsection{Quick-start (one-button) runner}\label{sec:runner}
\textbf{Script:} \texttt{run\_all.sh} builds CNFs for $5\le p\le 997$, solves SAT cases (printing a witness word if a SAT solver is available), and, for $p\ge 13$, emits DRAT proofs and checks them if \texttt{drat-trim} is installed.
\begin{lstlisting}
#!/usr/bin/env bash
set -euo pipefail

# Usage: ./run_all.sh [MAX_PRIME]
# Default MAX_PRIME is 997 if not provided.
MAXP="${1:-997}"

#prefer kissat for SAT speed if present; fall back to CaDiCaL.
SAT_SOLVER=""
if command -v kissat >/dev/null 2>&1; then
  SAT_SOLVER="kissat"
elif command -v cadical >/dev/null 2>&1; then
  SAT_SOLVER="cadical"
else
  echo "No SAT solver found (need kissat or cadical)"; exit 1
fi

#prefer CaDiCaL for UNSAT proof logging.
HAVE_CADICAL=0
command -v cadical >/dev/null 2>&1 && HAVE_CADICAL=1

HAVE_DRAT=0
command -v drat-trim >/dev/null 2>&1 && HAVE_DRAT=1

#generate primes 5..MAXP (simple sieve via Python).
PRIMES="$(python3 - "$MAXP" <<'PY'
import sys
MAX=int(sys.argv[1])
isp=[True]*(MAX+1)
if MAX>=0: isp[0]=False
if MAX>=1: isp[1]=False
import math
for i in range(2,int(math.isqrt(MAX))+1):
    if isp[i]:
        step=i
        start=i*i
        isp[start:MAX+1:step]=[False]*(((MAX-start)//step)+1)
print(" ".join(str(p) for p in range(5,MAX+1) if isp[p]))
PY
)"

SAT_P=()
UNSAT_P=()

for p in $PRIMES; do
  cnf=avoid_p${p}.cnf
  out=solver_p${p}.out
  echo "=== p=${p} ==="
  python3 make_cnf.py ${p} ${cnf}

  if (( p >= 13 )); then
    if (( HAVE_CADICAL )); then
      #solve with CaDiCaL and (attempt to) emit textual DRAT proof.
      #CaDiCaL syntax: cadical <cnf> <proof>
      cadical ${cnf} avoid_p${p}.drat > ${out} || true

      #check solver status from output.
      satline=$(grep -m1 -E '^s (SATISFIABLE|UNSATISFIABLE)' ${out} || true)
      if echo "$satline" | grep -q 'UNSAT'; then
        #verify DRAT if drat-trim is available.
        if (( HAVE_DRAT )); then
          if drat-trim ${cnf} avoid_p${p}.drat -q; then
            echo "DRAT verified for p=${p}"
            UNSAT_P+=("${p}")
          else
            echo "DRAT check FAILED for p=${p}"; exit 1
          fi
        else
          echo "drat-trim not found; skipped proof check for p=${p}"
          UNSAT_P+=("${p}")
        fi
      elif echo "$satline" | grep -q 'SATISFIABLE'; then
        echo "Unexpected SAT from CaDiCaL for p=${p}; extracting a model..."
        #get a model with the chosen SAT solver and verify it.
        ${SAT_SOLVER} -q ${cnf} > ${out}.sat || true
        satline2=$(grep -m1 -E '^s (SATISFIABLE|UNSATISFIABLE)' ${out}.sat || true)
        if echo "$satline2" | grep -q 'UNSAT'; then
          echo "WARNING: ${SAT_SOLVER} claims UNSAT too for p=${p}."
        else
          w=$(python3 model_to_word.py ${p} ${out}.sat)
          echo "witness p=${p}: ${w}"
          python3 verifier_strong_form.py ${p} "${w}"
          SAT_P+=("${p}")
        fi
      else
        echo "WARNING: could not parse solver status for p=${p} (see ${out})."
      fi
    else
      echo "WARNING: CaDiCaL not found; solving p=${p} without a proof."
      ${SAT_SOLVER} -q ${cnf} > ${out} || true
      satline=$(grep -m1 -E '^s (SATISFIABLE|UNSATISFIABLE)' ${out} || true)
      if echo "$satline" | grep -q 'SATISFIABLE'; then
        w=$(python3 model_to_word.py ${p} ${out})
        echo "witness p=${p}: ${w}"
        python3 verifier_strong_form.py ${p} "${w}"
        SAT_P+=("${p}")
      else
        echo "UNSAT (no proof logged) for p=${p}"
        UNSAT_P+=("${p}")
      fi
    fi

  else
    if [[ "${SAT_SOLVER}" == "kissat" ]]; then
      kissat -q ${cnf} > ${out} || true
    else
      cadical ${cnf} > ${out} || true
    fi
    satline=$(grep -m1 -E '^s (SATISFIABLE|UNSATISFIABLE)' ${out} || true)
    if echo "$satline" | grep -q 'UNSAT'; then
      echo "Unexpected UNSAT for p=${p}. Check ${out}."
      continue
    fi
    w=$(python3 model_to_word.py ${p} ${out})
    echo "witness p=${p}: ${w}"
    python3 verifier_strong_form.py ${p} "${w}"
    SAT_P+=("${p}")
  fi
done

echo
echo "==== SUMMARY ===="
if ((${#SAT_P[@]})); then
  echo "SAT primes (witness found): ${SAT_P[*]}"
else
  echo "SAT primes (witness found): none"
fi

if ((${#UNSAT_P[@]})); then
  echo "UNSAT primes (DRAT verified or declared): ${UNSAT_P[*]}"
else
  echo "UNSAT primes (DRAT verified or declared): none"
fi
\end{lstlisting}

\subsection{Cyclic sweep for $W_c(4,2)$ (moduli $M=13,\dots,34$) and per-modulus artifacts}\label{sec:wc42-sweep}

We include a small sweep to decide the cyclic van der Waerden value at $(k,r)=(4,2)$ by testing composite moduli $M\in\{13,\dots,34\}$. 
The CNF encoding adds clauses only for non-degenerate residue $4$-AP windows in $\mathbb{Z}_M$ (i.e., four distinct residues). 
The script prints a table and also writes per-modulus artifacts (CNF, witness/model or DRAT proof) to a results directory.

\paragraph{Commands.}
\begin{verbatim}
# Default sweep M=13..34, artifacts under ./results/
python3 find_wc42.py

# Choose a different output folder
python3 find_wc42.py --outdir out_wc42

# Prefer Kissat even if CaDiCaL is available (no DRAT)
python3 find_wc42.py --prefer-kissat
\end{verbatim}

\paragraph{Outputs.}
For each tested modulus $M$, the script writes to \texttt{<outdir>} (default \texttt{results/}):
\begin{itemize}
  \item \texttt{avoid\_M\textit{XX}.cnf} \hspace{0.35em} (DIMACS with non-degenerate windows only),
  \item if SAT: \texttt{model\_M\textit{XX}.txt} (raw solver output) and \texttt{witness\_M\textit{XX}.txt} (word over \texttt{B}/\texttt{R}),
  \item if UNSAT with CaDiCaL: \texttt{proof\_M\textit{XX}.drat} and \texttt{proof\_M\textit{XX}.drat.check.txt} (from \texttt{drat-trim}),
  \item summary tables \texttt{wc42\_results.csv} and \texttt{wc42\_results.tsv}.
\end{itemize}

\paragraph{Standalone composite-aware verifier.}
\texttt{verifier\_cyclic.py} re-checks any reported witness by testing only non-degenerate residue windows in $\mathbb{Z}_M$:
\begin{verbatim}
python3 verifier_cyclic.py 22 RRRBRRBRBBBRRRBRRBRBBB
python3 verifier_cyclic.py 33 BBBRBRRBRRRBBBRBRRBRRRBBBRBRRBRRR
\end{verbatim}

\medskip\noindent
\textbf{Script:} \texttt{find\_wc42.py}.
\begin{lstlisting}[language=Python]
#!/usr/bin/env python3
"""
Requires: cadical (preferred) or kissat, drat-trim preferred.
"""

import argparse, os, shutil, subprocess, sys
from typing import List, Tuple

#filesystem helpers

def ensure_dir(d: str):
    os.makedirs(d, exist_ok=True)

def write_text(path: str, text: str):
    with open(path, "w", encoding="utf-8") as f:
        f.write(text)

def append_text(path: str, text: str):
    with open(path, "a", encoding="utf-8") as f:
        f.write(text)

#system helpers

def have(cmd: str) -> bool:
    return shutil.which(cmd) is not None

def run(cmd: List[str], *, input_text: str = "") -> subprocess.CompletedProcess:
    return subprocess.run(cmd, input=input_text, text=True, capture_output=True, check=False)

def solver_status(stdout: str) -> str:
    for line in stdout.splitlines():
        if line.startswith("s "):
            if "UNSATISFIABLE" in line: return "UNSAT"
            if "SATISFIABLE" in line:   return "SAT"
    return "UNKNOWN"

#combinatorics

def nondeg_windows(M: int):
    for r in range(1, M):
        for i in range(M):
            w=[(i+k*r)%M for k in range(4)]
            if len(set(w))==4:
                yield tuple(w)

def is_valid_witness(word: str) -> bool:
    M = len(word)
    for a,b,c,d in nondeg_windows(M):
        block = [word[a], word[b], word[c], word[d]]
        if block.count('B')==4 or block.count('R')==4:
            return False
    return True

def build_dimacs_for_M(M: int) -> str:
    clauses=[]
    def var(i: int) -> int: return i+1
    for a,b,c,d in nondeg_windows(M):
        vs=[var(a),var(b),var(c),var(d)]
        clauses.append(vs)               # forbid BBBB
        clauses.append([-v for v in vs]) # forbid RRRR
    lines=[f"p cnf {M} {len(clauses)}"]
    for C in clauses: lines.append(" ".join(map(str,C))+" 0")
    return "\n".join(lines)+"\n"

def decode_model_strict(stdout: str, M: int) -> str | None:
    """Parse only 'v ' lines; return word over {B,R} or None if not found."""
    vals = {}
    saw = False
    model_lines=[]
    for line in stdout.splitlines():
        if line.startswith("v "):
            saw = True
            model_lines.append(line+"\n")
            for tok in line.split()[1:]:
                if tok == '0': continue
                if tok.lstrip('-').isdigit():
                    v=int(tok); vals[abs(v)] = (v>0)
    if not saw: return None
    word = "".join('R' if vals.get(i,False) else 'B' for i in range(1, M+1))
    return word

#solve one modulus

def solve_one(M: int, outdir: str, prefer_kissat: bool) -> Tuple[str, str, str]:
    """
    Returns (status, witness_or_note, proof_note)
    and writes artifacts into outdir.
    """
    ensure_dir(outdir)
    cnf_path   = os.path.join(outdir, f"avoid_M{M}.cnf")
    drat_path  = os.path.join(outdir, f"proof_M{M}.drat")
    check_path = os.path.join(outdir, f"proof_M{M}.drat.check.txt")
    model_path = os.path.join(outdir, f"model_M{M}.txt")
    wit_path   = os.path.join(outdir, f"witness_M{M}.txt")

    write_text(cnf_path, build_dimacs_for_M(M))

    have_cad   = have("cadical")
    have_kisat = have("kissat")
    have_drat  = have("drat-trim")

    use_cad = (have_cad and not prefer_kissat) or (have_cad and not have_kisat)

    if use_cad:
        proc = run(["cadical", cnf_path, drat_path])
        st = solver_status(proc.stdout)
        if st == "SAT":
            #save full stdout for provenance (includes v-lines)
            write_text(model_path, proc.stdout)
            word = decode_model_strict(proc.stdout, M)
            #fallback to kissat for clean model if needed
            if (word is None or not is_valid_witness(word)) and have_kisat:
                k = run(["kissat", "-q", cnf_path])
                if solver_status(k.stdout) == "SAT":
                    append_text(model_path, "\n# kissat stdout\n"+k.stdout)
                    w2 = decode_model_strict(k.stdout, M)
                    if w2 is not None: word = w2
            if word is None:
                return "SAT", "no model parsed", ""
            write_text(wit_path, word+"\n")
            note = ""
            if not is_valid_witness(word):
                note = "witness FAILS self-check"
            return "SAT", word if not note else f"{word}\t{note}", ""
        elif st == "UNSAT":
            proofnote = "no drat-trim"
            if have_drat:
                chk = run(["drat-trim", cnf_path, drat_path, "-q"])
                write_text(check_path, chk.stdout + chk.stderr)
                if ("s VERIFIED" in chk.stdout) or (chk.returncode == 0):
                    proofnote = "DRAT ok"
                else:
                    proofnote = "DRAT check failed"
            return "UNSAT", "", proofnote
        else:
            return "UNKNOWN", "", ""
    else:
        if not have_kisat:
            return "UNKNOWN", "", "no solver found"
        proc = run(["kissat", "-q", cnf_path])
        st = solver_status(proc.stdout)
        if st == "SAT":
            write_text(model_path, proc.stdout)
            word = decode_model_strict(proc.stdout, M)
            if word is None:
                return "SAT", "no model parsed", ""
            write_text(wit_path, word+"\n")
            note = ""
            if not is_valid_witness(word):
                note = "witness FAILS self-check"
            return "SAT", word if not note else f"{word}\t{note}", ""
        elif st == "UNSAT":
            return "UNSAT", "", ""  # no DRAT
        else:
            return "UNKNOWN", "", ""

#CLI

def main():
    ap = argparse.ArgumentParser(description="Sweep M to help determine W_c(4,2) with file outputs.")
    ap.add_argument("--start", type=int, default=13)
    ap.add_argument("--end", type=int, default=34)
    ap.add_argument("--outdir", type=str, default="results")
    ap.add_argument("--csv", type=str, default="wc42_results.csv")
    ap.add_argument("--tsv", type=str, default="wc42_results.tsv")
    ap.add_argument("--prefer-kissat", action="store_true",
                    help="Prefer Kissat even if CaDiCaL is available (no DRAT).")
    args = ap.parse_args()

    ensure_dir(args.outdir)
    csv_path = os.path.join(args.outdir, args.csv) if args.csv else ""
    tsv_path = os.path.join(args.outdir, args.tsv) if args.tsv else ""

    rows = []
    print("M\tStatus\tWitness/Note\tProof")
    for M in range(args.start, args.end + 1):
        st, w, pr = solve_one(M, args.outdir, args.prefer_kissat)
        print(f"{M}\t{st}\t{w}\t{pr}")
        rows.append((M, st, w, pr))

    if csv_path:
        write_text(csv_path, "M,Status,Witness,Proof\n")
        for M, st, w, pr in rows:
            wq  = f"\"{w}\""  if ("," in w)  else w
            prq = f"\"{pr}\"" if ("," in pr) else pr
            append_text(csv_path, f"{M},{st},{wq},{prq}\n")
        print(f"[wrote] {csv_path}")

    if tsv_path:
        write_text(tsv_path, "M\tStatus\tWitness\tProof\n")
        for M, st, w, pr in rows:
            append_text(tsv_path, f"{M}\t{st}\t{w}\t{pr}\n")
        print(f"[wrote] {tsv_path}")

    #friendly summary
    sat_33 = any(M == 33 and st == "SAT" for (M, st, _, _) in rows)
    unsat_34 = any(M == 34 and st.startswith("UNSAT") for (M, st, _, _) in rows)
    if sat_33 and unsat_34:
        print("\nSummary: SAT at M=33 and UNSAT at M=34 -> W_c(4,2)=34.\n")

if __name__ == "__main__":
    main()
\end{lstlisting}

\medskip\noindent
\textbf{Script:} \texttt{verifier\_cyclic.py}.
\begin{lstlisting}[language=Python]
#!/usr/bin/env python3
import sys

USAGE = "usage: verifier_cyclic.py <modulus M> <wordoverBR>"

def nondeg_windows(M):
    for r in range(1, M):
        for i in range(M):
            w = [(i + k*r) % M for k in range(4)]
            if len(set(w)) == 4:
                yield w

if __name__ == "__main__":
    if len(sys.argv) != 3:
        print(USAGE); sys.exit(2)
    M = int(sys.argv[1])
    w = sys.argv[2].strip().upper()
    assert len(w) == M and set(w) <= set("BR")
    for a,b,c,d in nondeg_windows(M):
        block = [w[a], w[b], w[c], w[d]]
        if block.count('B') == 4 or block.count('R') == 4:
            print("FAIL at", (a,b,c,d), "block=", "".join(block))
            sys.exit(1)
    print("OK")
\end{lstlisting}

\subsection{Word lists and manifest}\label{sec:manifest}
We include \texttt{solutions\_p5.txt}, \texttt{solutions\_p7.txt}, \texttt{solutions\_p11.txt} (one word per line). The file \texttt{artifact\_manifest.json} records filenames and SHA-256 hashes for reproducibility.

\paragraph{Acknowledgments.}
I thank Vito Valentino (ufl.edu) for pointing out the connection to cyclic van der Waerden numbers and for suggesting the finite sweep that led to the computation of $W_c(4,2)$.

\bibliographystyle{amsplain}
\bibliography{refs}

\providecommand{\bysame}{\leavevmode\hbox to3em{\hrulefill}\thinspace}
\providecommand{\MR}{\relax\ifhmode\unskip\space\fi MR }
\providecommand{\MRhref}[2]{%
  \href{http://www.ams.org/mathscinet-getitem?mr=#1}{#2}
}
\providecommand{\href}[2]{#2}
\begin{thebibliography}{10}

\bibitem{cadical}
Armin Biere, \emph{Cadical sat solver (official page/repository)}, \url{https://fmv.jku.at/cadical/} and \url{https://github.com/arminbiere/cadical}, 2020.

\bibitem{kissat}
\bysame, \emph{Kissat sat solver (official repository)}, \url{https://github.com/arminbiere/kissat}, 2020.

\bibitem{BurkertJohnson2011}
J.~Burkert and P.~Johnson, \emph{Szlam’s lemma: Mutant offspring of a euclidean ramsey problem from 1973, with numerous applications}, Ramsey Theory: Yesterday, Today, and Tomorrow (Alexander Soifer, ed.), Progress in Mathematics, vol. 285, Birkh{\"a}user, Boston, MA, 2011, pp.~97--113.

\bibitem{Chvatal1970}
V{\'a}{\v{s}}ek Chv{\'a}tal, \emph{Some unknown van der waerden numbers}, Combinatorial Structures and Their Applications (R.~K. Guy, H.~Hanani, N.~Sauer, and J.~Sch{\"o}nheim, eds.), Gordon and Breach, New York, 1970, pp.~31--33.

\bibitem{GRS1990}
R.~L. Graham, B.~L. Rothschild, and J.~H. Spencer, \emph{Ramsey theory}, 2 ed., Wiley–Interscience, 1990.

\bibitem{drattrim}
Marijn J.~H. Heule, \emph{Drat-trim (official repository)}, \url{https://github.com/marijnheule/drat-trim}, 2014.

\bibitem{Heule2016}
\bysame, \emph{The drat format and drat-trim checker}, \url{https://arxiv.org/abs/1610.06229}, 2016, arXiv:1610.06229.

\bibitem{LandmanRobertson2014}
B.~M. Landman and A.~Robertson, \emph{Ramsey theory on the integers}, 2 ed., Student Mathematical Library, vol.~73, American Mathematical Society, 2014.

\bibitem{Liber2025}
Benjamin Liber, \emph{On the independence numbers of the cyclic van der waerden hypergraphs}, 2025, Prop.~4.6 gives $W_c(4,2)\ge 13$.

\bibitem{LuPeng2012}
Linyuan Lu and Xing Peng, \emph{Monochromatic 4-term arithmetic progressions in 2-colorings of $\mathbb{Z}_n$}, Journal of Combinatorial Theory, Series A \textbf{119} (2012), no.~5, 1048--1065.

\bibitem{vdW1927}
B.~L. van~der Waerden, \emph{Beweis einer {Baudetschen} vermutung}, Nieuw Archief voor Wiskunde \textbf{15} (1927), 212--216 (German).

\bibitem{WetzlerHeuleHunt2014}
Nathan Wetzler, Marijn J.~H. Heule, and Warren A.~Jr. Hunt, \emph{Drat-trim: Efficient checking and trimming using expressive clausal proofs}, Theory and Applications of Satisfiability Testing -- SAT 2014, Lecture Notes in Computer Science, vol. 8561, Springer, 2014, pp.~422--429.

\bibitem{Wolf2010}
Julia Wolf, \emph{The minimum number of monochromatic 4-term progressions in $\mathbb{Z}_p$}, Journal of Combinatorics \textbf{1} (2010), no.~1, 53--68.

\end{thebibliography}

\end{document}